\documentclass[a4paper,11pt]{amsart}

\usepackage{hyperref} 
\usepackage[lite]{amsrefs} 
\usepackage{microtype} 
\usepackage{verbatim} 
\usepackage{color}  
\usepackage{amsmath,amssymb} 
\usepackage{tikz-cd}
\usepackage[all]{xy} 
\usepackage{enumerate} 
\usepackage{mathtools} 
\usepackage{mathrsfs}
\usepackage{bm}
\title[\textit{S}-arithmetic spinor groups]{\textit{S}-arithmetic spinor groups with the same finite quotients and distinct \(\bm{\ell^2}\)-cohomology}
\author{Holger Kammeyer}
\author{Roman Sauer}
\address{Institute for Algebra and Geometry\\ Karlsruhe Institute of Technology\\ Germany}
\email{holger.kammeyer@kit.edu}
\email{roman.sauer@kit.edu}
\urladdr{https://topology.math.kit.edu/}

\subjclass[2010]{20G30 (Primary), 20G25, 11F75 (Secondary)}
\keywords{L2-Betti numbers, profinite completion, S-arithmetic groups}
\thanks{This research was supported by the Deutsche Forschungsgemeinschaft (DFG), grant numbers SA 1661/4-1 and KA 4641/1-1.}

\newtheorem{theorem}{Theorem}
\newtheorem{corollary}{Corollary}
\newtheorem{proposition}{Proposition}
\newtheorem{lemma}{Lemma}

\theoremstyle{definition}

\theoremstyle{remark}
\newtheorem{remark}{Remark}

\makeatletter
   \let\c@corollary=\c@theorem
   \let\c@proposition=\c@theorem
   \let\c@lemma=\c@theorem
   \let\c@definition=\c@theorem
   \let\c@remark=\c@theorem
   \let\c@example=\c@theorem
   \let\c@equation=\c@theorem
   \let\c@conjecture=\c@theorem
   \let\c@question=\c@theorem
\makeatother

\newcommand*{\arXiv}[1]{ \href{http://www.arxiv.org/abs/#1}{arXiv:\textbf{#1}}}

\newcommand*{\N}{\mathbb N}
\newcommand*{\Z}{\mathbb Z}
\newcommand*{\Q}{\mathbb Q}
\newcommand*{\R}{\mathbb R}
\newcommand*{\C}{\mathbb C}

\newcommand{\calO}{\mathcal{O}}

\newcommand{\legendre}[2]{\genfrac{(}{)}{}{}{#1}{#2}}

\DeclareMathOperator{\rank}{rank}

\DeclareMathOperator{\im}{im}

\DeclareMathOperator{\ind}{ind}

\DeclareMathOperator{\ad}{ad}

\DeclareMathOperator{\Hom}{Hom}

\DeclareMathOperator{\GL}{GL}
\DeclareMathOperator{\id}{id}
\DeclareMathOperator{\bSO}{\mathbf{SO}}

\DeclareMathOperator{\bGL}{\mathbf{GL}}

\DeclareMathOperator{\bSpin}{\mathbf{Spin}}

\DeclareMathOperator{\bG}{\mathbf{G}}

\newcounter{commentcounter}

\usepackage{ifthen,srcltx}
\newcommand{\showcomments}{yes}

\newsavebox{\commentbox}
\newenvironment{com}%
{\ifthenelse{\equal{\showcomments}{yes}}%
{\footnotemark
        \begin{lrbox}{\commentbox}
        \begin{minipage}[t]{1.25in}\raggedright\sffamily\tiny
        \footnotemark[\arabic{footnote}]}
{\begin{lrbox}{\commentbox}}}
{\ifthenelse{\equal{\showcomments}{yes}}
{\end{minipage}\end{lrbox}\marginpar{\usebox{\commentbox}}}
{\end{lrbox}}}

\DefineSimpleKey{bib}{myurl}

\newcommand\myurl[1]{\url{#1}}

\BibSpec{webpage}{%
  +{}{\PrintAuthors} {author}
  +{,}{ \textit} {title}
  +{,}{ } {note}
  +{,}{ \myurl} {myurl}
  +{}{ \parenthesize} {date}
  +{.}{ } {transition}
}

\newcommand{\ignore}[1]{}

\begin{document}

\begin{abstract}
 In this note we refine examples by Aka from arithmetic to S-arithmetic groups to show that the vanishing of the $i$-th $\ell^2$-Betti number is not a profinite invariant for all $i\ge 2$. 
\end{abstract} 

\maketitle

\section{Introduction}

Let \(\mathcal{C}\) be a class of groups.  We call a group invariant \emph{profinite among \(\mathcal{C}\)} if it agrees for any two groups in \(\mathcal{C}\) which have the same set of finite quotients.  If \(\mathcal{C}\) is the class of all finitely generated \emph{residually finite} groups, we plainly say the invariant is \emph{profinite}.  

An example of a profinite invariant is the abelianization \(H_1(\Gamma)\), and hence also the first \emph{Betti number} \(b_1(\Gamma)\).  Bridson--Conder--Reid~\cite{Bridson-Conder-Reid:determining}*{Corollary~3.3} apply L\"uck's approximation theorem~\cite{Lueck:approximating} to conclude that the first \emph{\(\ell^2\)-Betti number} \(b^{(2)}_1(\Gamma)\) is profinite among finitely presented, residually finite groups; see \citelist{\cite{Lueck:l2-invariants} \cite{Kammeyer:introduction-to-l2}\cite{sauer-survey}} for an introduction to \(\ell^2\)-Betti numbers.  In contrast, Reid mentions in~\cite{Reid:profinite-properties}*{6.2, p.\,88} that examples of arithmetic groups with the same profinite completion constructed by Aka in~\cite{Aka:profinite-completions} have different higher \(\ell^2\)-Betti numbers, despite being finitely presented and residually finite.  It turns out that these groups actually provide counterexamples to the profiniteness of \(b^{(2)}_n\) for all even numbers \(n \ge 6\).  An expository account of the discussion so far can be found in the Master thesis of Stucki~\cite{Stucki:master-thesis}.

In this article, we refine the construction given in~\cite{Aka:profinite-completions} from arithmetic to \mbox{\(S\)-arithmetic} groups to see that for no \(n \ge 2\), the \(n\)-th \(\ell^2\)-Betti number is profinite among finitely presented, residually finite groups.

\begin{theorem} \label{thm:main-theorem}
  For \(k \ge 2\), fix distinct prime numbers \(p_1, \ldots, p_k \ge 89\) within the arithmetic progression \(17 +24 \N\) and set
  \[ \Gamma^k_\pm = \mathbf{Spin}(\langle \pm 1,\, \pm 1,\, \pm 1,\, \pm p_1 \cdots p_k,\, 3 \rangle)(\Z[\textstyle\frac{1}{p_1 \cdots p_k}]). \]
  Then \(\widehat{\Gamma^k_+} \cong \widehat{\Gamma^k_-}\), and \(b^{(2)}_n(\Gamma^k_+) > 0\) iff \(n = k\), whereas \(b^{(2)}_n(\Gamma^k_-) > 0\) iff \(n = k + 2\).
\end{theorem}

Of course, the result borrows from Dirichlet's theorem that the arithmetic progression \(17 + 24\N\) contains infinitely many primes. The precise meaning of $\Gamma^k_\pm$ is explained in Section~\ref{section:integral-spinor-groups}. The notation \(\widehat{\Gamma}\) refers to the \emph{profinite completion} of \(\Gamma\), the projective limit over the inverse system of all finite quotient groups of \(\Gamma\).  For finitely generated \(\Gamma\) and \(\Lambda\), it turns out that \(\widehat{\Gamma} \cong \widehat{\Lambda}\) if and only if \(\Gamma\) and \(\Lambda\) have the same set of finite quotients~\cite{Ribes-Zalesskii:profinite-groups}*{Corollary~3.2.8, p.\,89}.  The groups \(\Gamma^k_{\pm}\) are finitely presented as \(S\)-arithmetic subgroups of reductive algebraic groups~\cite{Platonov-Rapinchuk:algebraic-groups}*{Theorem~5.11, p.\,272}.  As such, they are also residually finite by Malcev's theorem~\cite{Nica:malcev-selberg}.

The proof of Theorem~\ref{thm:main-theorem} naturally splits into two parts: the construction of the isomorphism \(\widehat{\Gamma^k_+} \cong \widehat{\Gamma^k_-}\), and the calculation of the \(\ell^2\)-Betti numbers \(b^{(2)}_n(\Gamma^k_\pm)\).  For the first part, it is key to know that both the \(\Q\)-isotropic group \(\Gamma^k_-\) and the \(\Q\)-anisotropic group \(\Gamma^k_+\) of type \(B_2\) satisfy the \emph{congruence subgroup property (CSP)}.  For the sake of a self-contained treatment, Section~\ref{section:csp} summarizes the verification of CSP in this case because the version we apply was not entirely available when the standard references \cite{Margulis:discrete-subgroups} and \cite{Platonov-Rapinchuk:algebraic-groups} were written.  CSP reduces the construction of an isomorphism \(\widehat{\Gamma^k_+} \cong \widehat{\Gamma^k_-}\) to verifying that the spinor groups of the forms \(\langle \pm 1,\, \pm 1,\, \pm 1,\, \pm p_1 \cdots p_k, 3 \rangle\) are isomorphic over the \(p\)-adic integers \(\Z_p\).  As we verify in Section~\ref{section:integral-spinor-groups}, this is ensured because the quadratic forms are \(\Z_p\)-isometric.  At this point it enters that the primes \(p_i\) were so chosen that \(p_1, \ldots, p_k \equiv 1 \ (8)\).  

The second part of the proof is given in Section~\ref{section:quadratic-forms}.  We will explain that finding the nontrivial \mbox{\(\ell^2\)-cohomology} of \(\Gamma^k_\pm\) is essentially a matter of showing that the \(\Q_{p_i}\)-rank of the two spinor groups is equal to one.  The rank of a spinor group equals the Witt index of the quadratic form and it comes out of the Hasse--Minkowski theory of quadratic forms that the Witt index is one because the primes \(p_i\) were so chosen that \(p_1, \ldots, p_k \equiv 2 \ (3)\).  This explains how the arithmetic progression \(17 + 24\N\) comes into play: it consists precisely of the numbers that satisfy both relevant congruence relations according to the Chinese remainder theorem.

\section{CSP for spinor groups}
\label{section:csp}

Let \(k \subset \C\) be a number field and let \(\mathbf{G} \subset \mathbf{GL_n}\) be an (absolutely almost) simple and simply-connected linear algebraic \(k\)-group.  Let \(S\) be a finite set of places of \(k\) containing all the archimedean ones.  Recall that any place \(v\) defines a completion \(k_v\) which comes with a valuation ring \(\mathcal{O}_v\) if \(v\) is nonarchimedean.  The ring of \(S\)-integers in \(k\) is denoted by
\[ \mathcal{O}_{k, S} = \{ x \in k \colon v(x) \ge 0 \text{ for all } v \notin S \}.\]
We have a diagonal embedding of the group \(\mathbf{G}(\mathcal{O}_{k,S})\) into the \emph{adele points} \(\mathbf{G}(\mathbb{A}_{k,S})\) defined as the \emph{restricted product} \(\prod^r_{v \notin S} \mathbf{G}(k_v)\) consisting of elements with almost all coordinates in \(\mathbf{G}(\mathcal{O}_v)\).  The group \(\mathbf{G}(\mathbb{A}_{k, S})\) carries the locally compact colimit topology determined by the canonical embeddings of \(\prod_{v \in T \setminus S} \mathbf{G}(k_v)\times\prod_{v\not\in T}\mathbf{G}(\mathcal{O}_v)\) with \(|T| < \infty\).  The Kneser--Platonov \emph{strong approximation theorem} implies that the group \(\mathbf{G}(\mathcal{O}_{k,S})\) has closure \(\prod_{v \notin S} \mathbf{G}(\mathcal{O}_v)\) in \(\mathbf{G}(\mathbb{A}_{k,S})\), provided the group \(G_S = \prod_{v \in S} \mathbf{G}(k_v)\) is not compact.  The group \(\prod_{v \notin S} \mathbf{G}(\mathcal{O}_v)\) is compact and totally disconnected, thus \emph{profinite}.  The universal property of the profinite completion therefore yields a canonical epimorphism \(\widehat{\mathbf{G}(\mathcal{O}_{k,S})} \rightarrow \prod_{v \notin S} \mathbf{G}(\mathcal{O}_v)\).  The kernel \(C(k, \mathbf{G}, S)\) of this homomorphism is called the \emph{congruence kernel}.  It thus fits into a short exact sequence
\begin{equation} \label{eq:ses} 1 \longrightarrow C(k, \mathbf{G}, S) \longrightarrow \widehat{\mathbf{G}(\mathcal{O}_{k,S})} \longrightarrow \prod_{v \notin S} \mathbf{G}(\mathcal{O}_v) \longrightarrow 1. \end{equation}
The \(k\)-group \(\mathbf{G}\) is said to have the \emph{congruence subgroup property} with respect to \(S\) (in the rigorous sense) if \(C(k, \mathbf{G}, S)\) is trivial.  For more details and references up to this point, we refer to~\cite{Kammeyer:profinite-commensurability}*{Sections~2.2 and~2.3}.

The embeddings of \(\mathbf{G}(\mathcal{O}_{k,S})\) into \(\widehat{\mathbf{G}(\mathcal{O}_{k,S})}\) and \(\mathbf{G}(\mathbb{A}_{k,S})\) induce two a~priori different topologies on the \(S\)-arithmetic group \(\mathbf{G}(\mathcal{O}_{k,S})\).  A fundamental system of neighborhoods \(\mathcal{U}_a\) for the former consists of finite index normal subgroups.  For the latter, a fundamental system \(\mathcal{U}_c\) consists of the so-called \emph{principal congruence subgroups} defined by the kernels of the homomorphisms \(\mathbf{G}(\mathcal{O}_{k,S}) \rightarrow \mathbf{G}(\mathcal{O}_{k,S} / \mathfrak{a})\) given by reducing matrix entries mod~\(\mathfrak{a}\) for nonzero ideals \(\mathfrak{a} \subset \mathcal{O}_{k, S}\).  Since \(S\)-arithmetic groups are finitely generated, the congruence subgroup property thus says precisely that each finite index subgroup of \(\mathbf{G}(\mathcal{O}_{k,S})\) is a \emph{congruence subgroup}, meaning it contains a principal congruence subgroup.  This explains the terminology.

When setting up proofs of CSP, it is useful to shift attention from \(\mathbf{G}(\mathcal{O}_{k,S})\) to the \(k\)-rational points \(\mathbf{G}(k)\) by considering \(\mathcal{U}_a\) and \(\mathcal{U}_c\) as fundamental systems of unit neighborhoods in \(\mathbf{G}(k)\) rather than in \(\mathbf{G}(\mathcal{O}_{k,S})\).  These systems define the \emph{arithmetic topology} and the \emph{congruence topology} on \(\mathbf{G}(k)\), respectively.  We obtain associated completions with respect to the canonical uniform structures.  For the arithmetic topology, this completion is commonly denoted \(\widehat{\mathbf{G}(k)}\), accepting a slight abuse of notation.  For the congruence topology, the completion can be identified with \(\mathbf{G}(\mathbb{A}_{k,S})\) because the latter is a complete Hausdorff group which contains \(\mathbf{G}(k)\) densely by strong approximation.

Since \(\mathcal{U}_a\) is finer than \(\mathcal{U}_c\), we obtain a unique homomorphism \(\phi \colon \widehat{\mathbf{G}(k)} \rightarrow \mathbf{G}(\mathbb{A}_{k,S})\) which restricts to the identity on \(\mathbf{G}(k)\).  We observe that \(\widehat{\mathbf{G}(k)}\) contains \(\widehat{\mathbf{G}(\mathcal{O}_{k,S})}\) as the closure of \(\mathbf{G}(\mathcal{O}_{k,S})\), hence the image \(\im(\phi)\) contains the open group \(\prod_{v \notin S} \mathbf{G}(\mathcal{O}_v)\), and thus is itself open.  This shows that \(\im(\phi)\) is closed.  But it is moreover dense because it contains \(\mathbf{G}(k)\), so \(\phi\) is surjective. It is obvious that the congruence kernel lies in $\ker(\phi)$, and it is easy to verify that they coincide. 
Thus the congruence kernel also appears in the short exact sequence
\begin{equation} \label{eq:extension} 1 \longrightarrow C(k, \mathbf{G}, S) \xrightarrow{\ \chi \ } \widehat{\mathbf{G}(k)} \xrightarrow{ \ \varphi \ } \mathbf{G}(\mathbb{A}_{k, S}) \longrightarrow 1. \end{equation}
It is in this set-up that the congruence kernel reveals its secrets.  Considering the compact abelian group \(\R / \Z\) as trivial \(\mathbf{G}(\mathbb{A}_{k, S})\)-module, the group extension~\eqref{eq:extension} gives rise to a five term inflation-restriction exact sequence in continuous cohomology as in~\cite{Serre:local-fields}*{Remark, p.\,118}.  The first four terms are
\begin{align*}
  0 \rightarrow & H^1_{\textup{ct}}(\mathbf{G}(\mathbb{A}_{k,S}), \R / \Z) \xrightarrow{\varphi^*} H^1_{\textup{ct}}(\widehat{\mathbf{G}(k)}, \R / \Z) \xrightarrow{\chi^*} H^1_{\textup{ct}}(C(k, \mathbf{G}, S), \R / \Z)^{\mathbf{G}(\mathbb{A}_{k,S})} \\
  & \xrightarrow{\psi} H^2_{\textup{ct}}(\mathbf{G}(\mathbb{A}_{k,S}), \R / \Z).
\end{align*}
The morphisms \(\varphi^*\), \(\chi^*\), and \(\psi\) are called \emph{inflation}, \emph{restriction}, and \emph{transgression}, respectively.  The group \(H^1_{\textup{ct}}(\widehat{\mathbf{G}(k)}, \R / \Z)\) consists of continuous homomorphisms \(\widehat{\mathbf{G}(k)} \rightarrow \R / \Z\) to an abelian group, hence they are trivial on the closure of the commutator subgroup \([\widehat{\mathbf{G}(k)}, \widehat{\mathbf{G}(k)}]\).  Thus, if we assume
\begin{flalign} \label{eq:perfect}
  \text{the group } \mathbf{G}(k) \text{ of } k \text{-rational points is perfect,} \tag{A}
\end{flalign}
then \(H^1_{\textup{ct}}(\widehat{\mathbf{G}(k)}, \R / \Z)\) is trivial, implying \(\psi\) is injective.  If we moreover assume
\begin{flalign} \label{eq:central}
  \text{the extension of topological groups~\eqref{eq:extension} is central,} \tag{B}
\end{flalign}
then it is classified by an element \(e \in H^2_{\textup{ct}}(\mathbf{G}(\mathbb{A}_{k,S}), C(k, \mathbf{G}, S))\) and the group \(H^1_{\textup{ct}}(C(k, \mathbf{G}, S), \R / \Z)^{\mathbf{G}(\mathbb{A}_{k,S})}\) is just the \emph{Pontryagin dual} \(\overline{C(k, \mathbf{G}, S)} = \Hom(C(k, \mathbf{G}, S), \R / \Z)\) of \(C(k, \mathbf{G}, S)\).  A homomorphism \(f \in \overline{C(k, \mathbf{G}, S)}\) defines a change of coefficients homomorphism
\[ f_* \colon H^2_{\textup{ct}}(\mathbf{G}(\mathbb{A}_{k,S}), C(k, \mathbf{G}, S)) \longrightarrow H^2_{\textup{ct}}(\mathbf{G}(\mathbb{A}_{k,S}), \R / \Z) \]
and transgression has the description \(\psi(f) = f_* e\).  The fact that the extension \eqref{eq:extension} splits over \(\mathbf{G}(k)\) thus says that \(\im \psi\) lies within
\[ \ker \left(H^2_{\textup{ct}}(\mathbf{G}(\mathbb{A}_{k,S}), \R/\Z) \xrightarrow{\textup{restriction}} H^2(\mathbf{G}(k), \R / \Z) \right). \]
Here in the codomain we use ordinary group cohomology which is the same as continuous cohomology when \(\mathbf{G}(k)\) carries the discrete topology.  The above kernel is known as the \emph{metaplectic kernel} \(M(k, \mathbf{G}, S)\).  Thus if
\begin{flalign} \label{eq:metaplectic-trivial}
  \text{the metaplectic kernel } M(k, \mathbf{G}, S) \text{ is trivial,} \tag{C}
\end{flalign}
then so is \(C(k, \mathbf{G}, S)\).  Here is a summary of our discussion.

\begin{theorem}
  Suppose assumptions \eqref{eq:perfect}, \eqref{eq:central}, and \eqref{eq:metaplectic-trivial} hold true.  Then the \(k\)-group \(\mathbf{G}\) has the congruence subgroup property with respect to \(S\).
\end{theorem}

\begin{theorem} \label{thm:spin-csp}
  Let \(q\) be a quadratic form over a number field \(k\) and let \(S\) be a finite set of places of \(k\) containing all the archimedean ones.  Suppose
  \begin{enumerate}[(i)]
  \item the form \(q\) has at least five variables, \label{item:five-variables}
  \item the set \(S\) contains at least one nonarchimedean place, \label{item:nonarchimedean-place}
  \item the sum of the Witt indices of \(q\) over \(k_v\) for \(v \in S\) is at least two. \label{item:witt-index}
  \end{enumerate}
  Then the \(k\)-group \(\mathbf{Spin}(q)\) has CSP with respect to \(S\).
\end{theorem}

\begin{proof}
  We have to verify the three assertions from above.

\smallskip \noindent \eqref{eq:perfect}: It is true more generally that \(\mathbf{Spin}(k)\) has no non-central normal subgroups.  This result, due to Kneser~\cite{Kneser:orthogonale-gruppen}*{Satz~C}, is one of the earliest in this circle of ideas. Here assumption~\eqref{item:five-variables} enters.

\smallskip \noindent \eqref{eq:central}: This is again due to Kneser~\cite{Kneser:normalteiler}*{Satz~11.1} and dates back more than two decades later.  Beware that the cited theorem assumes \(q\) has \(\ge 8\) variables but the argument can easily be adapted to \(\ge 5\) variables, as was noticed by Vaserstein, [ibid., Zusatz~1]. Here assumptions~\eqref{item:five-variables} and~\eqref{item:witt-index} enter.

\smallskip \noindent \eqref{eq:metaplectic-trivial}: This follows from the definite solution of the metaplectic problem due to Prasad--Rapinchuk~\cite{Prasad-Rapinchuk:computation}*{Main Theorem} another two decades later.  It says in particular that \(M(k, \mathbf{G}, S)\) is trivial if \(S\) contains a finite place \(v\) such that \(\mathbf{G}\) is \(k_v\)-isotropic.  So we only need to recall that quadratic forms in five or more variables are isotropic at any finite place~\cite{Lam:quadratic-forms}*{Theorem~2.12, p.\,158} and apply Lemma~\ref{lemma:rank-equals-witt-index} below.  Here assumptions~\eqref{item:five-variables} and~\eqref{item:nonarchimedean-place} enter.
\end{proof}

\begin{remark}
  We stress that the main value of Theorem~\ref{thm:spin-csp} is that \(\mathbf{Spin}(q)\) is allowed to be \(k\)-anisotropic.  Assuming that \(q\) and thus \(\mathbf{Spin}(q)\) is \(k\)-isotropic, earlier proofs of \eqref{eq:perfect} \citelist{\cite{Dieudonne:geometrie} \cite{Eichler:quadratische-formen}}, \eqref{eq:central} \citelist{\cite{Raghunathan:on-csp} \cite{Vaserstein:structure}}, and \eqref{eq:metaplectic-trivial} \citelist{\cite{Kneser:normalteiler} \cite{Prasad-Raghunathan:metaplectic-kernel}} are available.
\end{remark}

Let us denote the two quadratic forms occurring in Theorem~\ref{thm:main-theorem} by \(q^k_{\pm}\).

\begin{corollary} \label{cor:spin-csp}
For each \(k \ge 2\), the \(\Q\)-groups \(\mathbf{Spin}(q^k_+)\) and \(\mathbf{Spin}(q^k_-)\) have the congruence subgroup property with respect to \(S = \{\infty, p_1, \ldots, p_k\}\).
\end{corollary}

\begin{proof}
  This is immediate from Theorem~\ref{thm:spin-csp} noting that \(q^k_{\pm}\) has Witt index at least one over \(\Q_{p_i}\) for \(i = 1, \ldots, k\) again by~\cite{Lam:quadratic-forms}*{Theorem~2.12, p.\,158}\footnote{In Proposition~\ref{prop: index one} we will show that the Witt index of \(q^k_{\pm}\) over \(\Q_{p_i}\) is precisely one.}. 
\end{proof}

\section{Integral spinor groups and profinite completions}
\label{section:integral-spinor-groups}

Let $k$ be a field of characteristic zero and let $(V,q)$ be a \emph{quadratic space}: a finite-dimensional $k$-vector space equipped with a nondegenerate quadratic form~$q$. If $(e_1,\dots,e_n)$ is the standard basis of $V=k^n$ then the quadratic form with $q(e_i)=a_i$ is denoted by $\langle a_1,\dots,a_n\rangle$. 

In the following, unadorned tensor products are taken over~$k$. 

The \emph{Clifford algebra} $C(V,q)$ is the quotient of the tensor algebra $T(V)$ by the two-sided ideal $I$ generated by the elements of the form $x\otimes x-q(x)$ for $x\in V$. The tensor algebra $T(V)=T_0(V)\oplus T_1(V)$ is graded according to the parity of length of tensors. Since 
$I$ is a homogeneous ideal, the Clifford algebra inherits a grading $C(V,q)=C_0(V,q)\oplus C_1(V,q)$. 
The obvious map $V\to C(V,q)^\circ$ extends to a homomorphism 
\[ \ast\colon C(V,q)\to C(V,q)^\circ,\]
where the latter denotes the opposite Clifford algebra. 
We have $(x_1\dots x_n)^\ast=x_n\dots x_1$, so $\ast$ is an involution on the Clifford algebra. 
Let \[\ad\colon C(V,q)^\times\to \GL(C(V,q))\] be the conjugation homomorphism sending $g$ to $x\mapsto gxg^{-1}$. 
Let $A$ be a commutative $k$-algebra. Then $\ad$ induces a homomorphism $\ad_A\colon (A\otimes C(V,q))^\times\to \GL(A\otimes C(V,q))$. The involution also extends to $A\otimes C(V,q)$. 
\begin{remark}
Assume that $\ad_A(g)$ with $g\in (A\otimes C_0(V,q))^\times$ preserves the subspace $A\otimes V\subset A\otimes C(V,q)$. Then 
    \[ gxg^{-1}=(gxg^{-1})^\ast=(g^\ast)^{-1}xg^\ast \]
    for every $x\in A\otimes V$. So $g^\ast g$ is in the center $Z(A\otimes C(V,q))$ for which 
    we have $A\otimes Z(C(V,q))=Z(A\otimes C(V,q))$. 
    But $Z(C(V,q))\cap C_0(V,q)=k$~\cite{milne}*{Proposition~24.58 on p.~531} and hence $g^\ast g\in A^\times$. 
\end{remark}
    The \emph{spinor group} for $(V,q)$ is the functor that assigns to a $k$-algebra $A$ the 
    group 
    \[ \bSpin(q)(A)=\bigl\{g\in (A\otimes C_0(V,q))^\times\mid \ad_A(g)(A\otimes V)=A\otimes V, g^\ast g=1_A\bigr\}.\]
    See~\cite{milne}*{Section~24 i} for more information on spinor groups. 
The spinor group is an almost simple, simply connected $k$-group. 
The morphism
\begin{equation} \label{eq:covering-map} \pi\colon \bSpin(q)\to\bSO(q) \end{equation}
given by the restricting $\ad_A$ to $A\otimes V$ is a central isogeny. 
As a $k$-group $\bSpin(q)$ has an embedding into a general linear group but we want to have a closer look at a particular one next. 

Let $\calO$ be a domain of characteristic zero. Let $L$ be a finitely 
generated free $\calO$-module and $q_L\colon L\times L\to \calO$ 
a quadratic form. Let $k$ be the fraction field of $\calO$, $V=k\otimes_\calO L$ and $q\colon V\times V\to k$ the induced quadratic form. 

Let $\Lambda L$ be the exterior algebra over $L$, which is a free $\calO$-module of rank $2^{\rank_\calO (L)}$. 
It is obvious that $\Lambda L\subset \Lambda V$ is an $\calO$-lattice in $\Lambda V$. Further, the \emph{symmetrization homomorphism} 
$\operatorname{sym}\colon \Lambda V\to C(V,q)$ whose restriction on 
$\Lambda^r V$ is given by 
\[ x_1\wedge \dots\wedge x_r\mapsto \frac{1}{r!}\sum_{\sigma\in \Sigma(r)}\operatorname{sign}(\sigma) x_{\sigma(1)}\otimes\dots\otimes x_{\sigma(r)}\] 
is an isomorphism of $k$-vector spaces~\cite{lawson+michelson}*{Proposition~1.3 on p.~10} and thus embeds $\Lambda L$ as an $\calO$-lattice into $C(V,q)$. 
So every $\calO$-basis of $L$ yields an isomorphism of 
the $k$-group $\bG(A)=GL(A\otimes C(V,q))$ with $\bGL_{2^d}$ where $d=2^{\dim V}$ and an isomorphism $GL_{2^d}(\calO)\cong GL(\Lambda L)$. Right multiplication by elements in $(A\otimes C(V,q))^\times$ on $C(V,q)$ induces a $k$-embedding \[j\colon \bSpin(q)\to \bG\cong \bGL_{2^d}.\] 
 Without referring to the theory of group schemes over arbitrary rings, we can now simply \emph{define} the $\calO$-points of $\bSpin(q)$ as 
\begin{equation}\label{eq: integral points of spinor groups}
\bSpin(q_L)(\calO):=\im j(k)\cap GL(\operatorname{sym}(\Lambda L))=\im j(k)\cap GL_{2^d}(\calO).
\end{equation}

The above definition is functorial with respect to isomorphisms of quadratic spaces. 

\begin{lemma} \label{lemma:isomorphic-spin-groups}
Let $(L,q_L)$ and $(L',q_{L'})$ be isomorphic quadratic spaces over $\calO$. Then the associated quadratic spaces $(V,q)$ and $(V',q')$ over $k$ are isomorphic and we obtain a commutative square with 
vertical group isomorphisms:
\[\begin{tikzcd}
\bSpin(q_L)(\calO)\arrow[d,"\cong"]\arrow[r,hook] & \bSpin(q)(k)\arrow[d, "\cong"]\\
\bSpin(q_{L'})(\calO)\arrow[r,hook] & \bSpin(q)(k)
\end{tikzcd}\]
\end{lemma}

\begin{proof} The symmetrization isomorphism is obviously functorial. So 
an isomorphism $(L,q_L)\cong (L',q_{L'})$ induces a commutative diagram  
\[\begin{tikzcd}
\Lambda L\arrow[r, hook]\arrow[d,"\cong"] & \Lambda V\arrow[r, "\cong"]\arrow[d, "\cong"] & C(V,q)\arrow[d,"\cong"]\\
\Lambda L'\arrow[r, hook] & \Lambda V'\arrow[r, "\cong"] & C(V',q')
\end{tikzcd}\]
which clearly implies the lemma. 
\end{proof}

\begin{lemma} \label{lemma:isometric-forms}
  Let \(p_1, \ldots, p_k\) be primes congruent to \(1\) mod \(8\).  Then the quadratic forms \(\langle 1, \ 1, \ 1, \ p_1 \cdots p_k \rangle\) and \(\langle -1, -1, -1, -p_1 \cdots p_k \rangle\) are isometric over \(\Z_p\) for all primes \(p\).
\end{lemma}

\begin{proof}
  If \(p \equiv 1 \ (4)\), the assertion follows because \(-1\) is a square in \(\Z_p\).  If \(p \equiv 3 \ (4)\), then both forms have the same unit discriminant and thus are \(\Z_p\)-isometric by \cite{Cassels:rational-quadratic-forms}*{Corollary to Theorem~3.1, p.\,116}.  Finally, for the case \(p = 2\), recall that  a unit \(x \in \Z_2^\ast\) is a square if and only if \(x \equiv 1 \ (8)\) under the canonical reduction \(\Z_2 \rightarrow \Z / 8 \Z\).  Therefore it is enough to see that \(\langle 1, \ 1, \ 1, \ 1 \rangle\) is \(\Z_2\)-isometric to \(\langle -1, -1, -1, -1 \rangle\).  Since \(\sqrt{-7} \in \Z_2\), the matrix
\[
  \begin{pmatrix}
    2 & 1 & 1 & \sqrt{-7} \\
    -1 & 2 & -\sqrt{-7} & 1 \\
    -1 & \sqrt{-7} & 2 & -1 \\
    -\sqrt{-7} & -1 & 1 & 2
  \end{pmatrix}
\]
defines an explicit such isometry (cf.~\cite{Aka:profinite-completions}*{Lemma~2}).
\end{proof}

Recall that we denoted the two quadratic forms in Theorem~\ref{thm:main-theorem} by \(q^k_\pm\).  Lemmas~\ref{lemma:isomorphic-spin-groups} and~\ref{lemma:isometric-forms} thus combine as follows.

\begin{proposition} \label{prop:isomorphic-spin-groups}
  For each prime \(p\), we have \(\mathbf{Spin}(q^k_+)(\Z_p) \cong \mathbf{Spin}(q^k_-)(\Z_p)\).
\end{proposition}

Let us remark at this point that we do not have to distinguish between abstract and topolgical isomorphisms of profinite groups in this article: all occurring profinite groups are topologically finitely generated which implies that isomorphisms are automatically continuous by a deep result of Nikolov--Segal~\cite{Nikolov-Segal:profinite-groups}.  We are now in a position to prove the first part of Theorem~\ref{thm:main-theorem}.

\begin{proposition} \label{prop:isomorphism}
  For each \(k \ge 2\), we have an isomorphism \(\widehat{\Gamma^k_+} \cong \widehat{\Gamma^k_-}\).
\end{proposition}

\begin{proof}
  According to Corollary~\ref{cor:spin-csp} and the short exact sequence in \eqref{eq:ses}, the profinite completions of \(\mathbf{Spin}(q^k_\pm)(\Z[\textstyle\frac{1}{p_1 \cdots p_k}])\) are given by the products
  \[ \prod_{p \,\nmid\, p_1 \cdots p_k} \mathbf{Spin}(q^k_\pm)(\Z_p) \]
  whose factors are isomorphic by Proposition~\ref{prop:isomorphic-spin-groups}.
\end{proof}

\section{Witt indices of quadratic forms and \(\ell^2\)-cohomology}
\label{section:quadratic-forms}

Let $(V,q)$ be a nondegenerate quadratic space over a field $k$. By the Witt decomposition theorem~\cite{witt} the quadratic space $(V,q)$ decomposes uniquely into an orthogonal sum 
\[ (V,q)\cong w(\mathbb{H}, q_\mathbb{H}) \perp (V_a, q_a)\] 
where $(V_a, q_a)$ is anisotropic and the left summand is a $w$-fold sum of hyperbolic planes $q_\mathbb{H}$, coming from the bilinear form $(x,y)\mapsto xy$.  The number $w$ is the \emph{Witt index} \(\ind_k (V,q)\).  Recall the setting of Theorem~\ref{thm:main-theorem}.

\begin{proposition}\label{prop: index one}
  For each \(i = 1, \ldots, k\), we have \(\ind_{\Q_{p_i}} q^k_\pm = 1\).
\end{proposition}

\begin{proof}
  By Lemma~\ref{lemma:isometric-forms} it is enough to consider the form \(q^k_+\).  Using that \(\sqrt{-1} \in \Q_{p_i}\), we have an orthogonal decomposition
  \[ q^k_+ = \mathbb{H} \oplus \langle 1, p_1 \cdots p_k, 3 \rangle \]
  and it remains to show that \(\langle 1, p_1 \cdots p_k, 3 \rangle\) is \(\Q_{p_i}\)-anisotropic.  According to~\cite{Serre:arithmetic}*{Theorem~6\,(ii), p.\,36}, this is granted precisely if the Hasse invariant \(\varepsilon\) of \(\langle 1, p_1 \cdots p_k, 3 \rangle\) differs from the Hilbert symbol \((-1,-3 p_1 \cdots p_k)_{\Q_{p_i}}\).  The Hilbert symbol equals \(1\) because \(-1\) is a square in \(\Q_{p_i}\).  The Hasse invariant is given by
  \[ \varepsilon = (1, p_1 \cdots p_k)_{\Q_{p_i}} (1, 3)_{\Q_{p_i}} (p_1 \cdots p_k, 3)_{\Q_{p_i}} = (p_1 \cdots p_k, 3)_{\Q_{p_i}} = (p_i, 3)_{\Q_{p_i}} \]
  and according to \cite{Serre:arithmetic}*{Theorem~1, p.\,20}, the latter is equal to the Legendre symbol \(\legendre{3}{p_i}\) which we can evaluate as
  \[ \legendre{3}{p_i} = \legendre{p_i}{3} = \legendre{2}{3} = -1 \]
  by quadratic reciprocity because \(p_i \equiv 1 \ (4)\) and because \( p_i \equiv 2 \ (3)\).  
\end{proof}

\begin{lemma} \label{lemma:rank-equals-witt-index}
  The $k$-rank of $\bSpin(q)$ equals the Witt index \(\ind_k (V,q)\).
\end{lemma}

\begin{proof}
The $k$-rank of $\bSpin(q)$ is the $k$-rank of $\bSO(q)$ since the homomorphism~$\pi$ from~\eqref{eq:covering-map} is a central isogeny~\cite{Margulis:discrete-subgroups}*{Corollary~1.4.6 on p.~42}. The maximal split torus $S$ of the latter is 
\[ S=\mathbb{G}_m\times\dots \times\mathbb{G}_m\times \{\id_{V_a}\}< \bSO(q_\mathbb{H})^w\times \bSO(q_a)<\bSO(q)\]
and hence $m$-dimensional~\cite{milne}*{Example~25.5, p.\,545}. 
\end{proof}

\begin{corollary} \label{cor:local-rank-one}
  For each \(i=1, \ldots, k\), the \(\Q_{p_i}\!\)-rank of \(\mathbf{Spin}(q^k_\pm)\) is one.
\end{corollary}

Finally, we have collected all preliminaries to give the proof of Theorem~\ref{thm:main-theorem}.

\begin{proof}[Proof of Theorem~\ref{thm:main-theorem}.]
  In view of Proposition~\ref{prop:isomorphism}, it remains to find the positive \(\ell^2\)-Betti numbers of \(\Gamma^k_\pm\).  To this end, we note that the \(S\)-arithmetic groups \(\Gamma^k_\pm\) are lattices in the locally compact groups~\cite{Margulis:discrete-subgroups}*{Theorem 3.2.4 on p.~63}
  \[ G^k_{S, \pm} = \prod_{v \in S} \mathbf{Spin}(q^k_\pm)(\Q_v) \]
  where \(S = \{\infty, p_1, \ldots, p_k\}\) and \(\Q_\infty = \R\).  For a lattice \(\Gamma \le G\) in a locally compact (second countable) group \(G\) with a fixed Haar measure \(\mu\), a theorem of Kyed--Petersen--Vaes~\cite{Kyed-Petersen-Vaes:l2-betti-locally-compact} shows that
  \[ b^{(2)}_r(\Gamma) = b^{(2)}_r(G, \mu) \cdot \mu(G / \Gamma), \]
  where \(\mu(G / \Gamma)\) is the induced \(G\)-invariant measure and \(b^{(2)}_r(G, \mu)\) is the \emph{\(r\)-th \(\ell^2\)-Betti number of the locally compact group} \(G\) defined by Petersen~\cite{Petersen:phd-thesis}.  Suppressing \(\mu\) from the notation, we thus only have to find out for which~\(r\) we have \(b^{(2)}_r(G^k_{S, \pm}) > 0\).  Petersen's K\"unneth formula~[ibid., Theorems~6.5 and~6.7] shows that \(b^{(2)}_r(G^k_{S, \pm})\) is given by
  \[ \sum_{r = s_0 + \cdots + s_k} b^{(2)}_{s_0} (\mathbf{Spin}(q^k_\pm)(\R))\; b^{(2)}_{s_1} (\mathbf{Spin}(q^k_\pm)(\Q_{p_1}))\; \cdots \; b^{(2)}_{s_k} (\mathbf{Spin}(q^k_\pm)(\Q_{p_k})). \]
  We note that the group \(\mathbf{Spin}(q^k_-)(\R)\) is a two-fold cover of the oriented isometry group of hyperbolic \(4\)-space.  Therefore \(b^{(2)}_s(\mathbf{Spin}(q^k_-)(\R))\) is positive if and only if \(s = 2\) by a result of Dodziuk~\cite{Dodziuk:rotationally}.  The group \(\mathbf{Spin}(q^k_+)(\R)\) is compact so that \(b^{(2)}_s(\mathbf{Spin}(q^k_+)(\R))\) is positive if and only if \(s = 0\) by~\cite{Petersen:phd-thesis}*{Propositions~3.4 and~3.6}.  For \(i = 1, \ldots, k\), we have that \(b^{(2)}_s(\mathbf{Spin}(q^k_\pm)(\Q_{p_i}))\) is positive if and only if \(s = \rank_{\Q_{p_i}} \mathbf{Spin}(q^k_\pm)\).  This follows from the calculations of Dymara--Januszkiewicz~\cite{Dymara-Januszkiewicz:cohomology-of-buildings} as was observed by Petersen--Valette~\cite{Petersen-Valette:l2-betti-plancherel}*{Corollary~13}.  For technical reasons, it is assumed in these considerations that the residue field \(\mathbb{F}_{p_i}\) of \(\Q_{p_i}\) has large enough characteristic.  We ensured this by requiring \(p_i \ge 89\) though it seems likely that this condition is not needed.

  The discussion reveals that the above sum has at most one nonvanishing summand and Corollary~\ref{cor:local-rank-one} implies that such a summand occurs for \(q^k_+\) if and only if \(r = k\), whereas it occurs for \(q^k_-\) if and only if \(r = k + 2\).
\end{proof}


\begin{bibdiv}[References]
  \begin{biblist}
    
\bib{Aka:profinite-completions}{article}{
   author={Aka, M.},
   title={Profinite completions and Kazhdan's property (T)},
   journal={Groups Geom. Dyn.},
   volume={6},
   date={2012},
   number={2},
   pages={221--229},
}


\bib{Bridson-Conder-Reid:determining}{article}{
   author={Bridson, M. R.},
   author={Conder, M. D. E.},
   author={Reid, A. W.},
   title={Determining Fuchsian groups by their finite quotients},
   journal={Israel J. Math.},
   volume={214},
   date={2016},
   number={1},
   pages={1--41},
}

\bib{Cassels:rational-quadratic-forms}{book}{
   author={Cassels, J. W. S.},
   title={Rational quadratic forms},
   series={London Mathematical Society Monographs},
   volume={13},
   publisher={Academic Press, Inc. [Harcourt Brace Jovanovich, Publishers],
   London-New York},
   date={1978},
   pages={xvi+413},
}

\bib{Dodziuk:rotationally}{article}{
   author={Dodziuk, J.},
   title={$L^{2}$\ harmonic forms on rotationally symmetric Riemannian
   manifolds},
   journal={Proc. Amer. Math. Soc.},
   volume={77},
   date={1979},
   number={3},
   pages={395--400},
}

\bib{Dieudonne:geometrie}{book}{
   author={Dieudonn\'e, J.},
   title={La g\'eom\'etrie des groupes classiques},
   language={French},
   series={Ergebnisse der Mathematik und ihrer Grenzgebiete (N.F.), Heft 5},
   publisher={Springer-Verlag, Berlin-G\"ottingen-Heidelberg},
   date={1955},
   pages={vii+115},
}

\bib{Dymara-Januszkiewicz:cohomology-of-buildings}{article}{
   author={Dymara, J.},
   author={Januszkiewicz, T.},
   title={Cohomology of buildings and their automorphism groups},
   journal={Invent. Math.},
   volume={150},
   date={2002},
   number={3},
   pages={579--627},
}

\bib{Eichler:quadratische-formen}{book}{
   author={Eichler, M.},
   title={Quadratische Formen und orthogonale Gruppen},
   language={German},
   series={Die Grundlehren der mathematischen Wissenschaften in
   Einzeldarstellungen mit besonderer Ber\"ucksichtigung der
   Anwendungsgebiete. Band LXIII},
   publisher={Springer-Verlag, Berlin-G\"ottingen-Heidelberg},
   date={1952},
   pages={xii+220},
}

%
\bib{Kammeyer:introduction-to-l2}{book}{
     author={Kammeyer, H.},
     title={Introduction to \(\ell^2\)-invariants},
     series={lecture notes},
     note={Available for download at \url{http://topology.math.kit.edu/21_679.php}},
     date={2017/2018},
}

\bib{Kammeyer:profinite-commensurability}{article}{
  author={Kammeyer, H.},
  title={Profinite commensurability of S-arithmetic groups},
  note={e-print \arXiv{1802.08559}},
  year={2018},
}

\bib{Kneser:normalteiler}{article}{
   author={Kneser, M.},
   title={Normalteiler ganzzahliger Spingruppen},
   language={German},
   journal={J. Reine Angew. Math.},
   volume={311/312},
   date={1979},
   pages={191--214},
}

\bib{Kneser:orthogonale-gruppen}{article}{
   author={Kneser, M.},
   title={Orthogonale Gruppen \"uber algebraischen Zahlk\"orpern},
   language={German},
   journal={J. Reine Angew. Math.},
   volume={196},
   date={1956},
   pages={213--220},
}

\bib{Kyed-Petersen-Vaes:l2-betti-locally-compact}{article}{
   author={Kyed, D.},
   author={Petersen, H.\,D.},
   author={Vaes, S.},
   title={$L^2$-Betti numbers of locally compact groups and their cross
   section equivalence relations},
   journal={Trans. Amer. Math. Soc.},
   volume={367},
   date={2015},
   number={7},
   pages={4917--4956},
}

\bib{Lam:quadratic-forms}{book}{
   author={Lam, T. Y.},
   title={Introduction to quadratic forms over fields},
   series={Graduate Studies in Mathematics},
   volume={67},
   publisher={American Mathematical Society, Providence, RI},
   date={2005},
   pages={xxii+550},
}

\bib{lawson+michelson}{book}{
   author={Lawson, H. Blaine, Jr.},
   author={Michelsohn, Marie-Louise},
   title={Spin geometry},
   series={Princeton Mathematical Series},
   volume={38},
   publisher={Princeton University Press, Princeton, NJ},
   date={1989},
   pages={xii+427},
}

%
\bib{Lueck:approximating}{article}{
   author={L\"uck, W.},
   title={Approximating $L^2$-invariants by their finite-dimensional
   analogues},
   journal={Geom. Funct. Anal.},
   volume={4},
   date={1994},
   number={4},
   pages={455--481},
}

\bib{Lueck:l2-invariants}{book}{
   author={L\"uck, W.},
   title={$L^2$-invariants: theory and applications to geometry and
   $K$-theory},
   series={Ergebnisse der Mathematik und ihrer Grenzgebiete. 3. Folge. A
   Series of Modern Surveys in Mathematics},
   volume={44},
   publisher={Springer-Verlag, Berlin},
   date={2002},
   pages={xvi+595},
}

\bib{Margulis:discrete-subgroups}{book}{
   author={Margulis, G. A.},
   title={Discrete subgroups of semisimple Lie groups},
   series={Ergebnisse der Mathematik und ihrer Grenzgebiete (3) [Results in
   Mathematics and Related Areas (3)]},
   volume={17},
   publisher={Springer-Verlag, Berlin},
   date={1991},
   pages={x+388},
}
\bib{milne}{book}{
   author={Milne, J. S.},
   title={Algebraic groups},
   series={Cambridge Studies in Advanced Mathematics},
   volume={170},
   note={The theory of group schemes of finite type over a field},
   publisher={Cambridge University Press, Cambridge},
   date={2017},
   pages={xvi+644},
}

\bib{Nica:malcev-selberg}{article}{
  author={Nica, B.},
  title={Linear groups - Malcev's theorem and Selberg's lemma},
  note={e-print \arXiv{1306.2385}},
  year={2013},
}

\bib{Nikolov-Segal:profinite-groups}{article}{
   author={Nikolov, N.},
   author={Segal, D.},
   title={On finitely generated profinite groups. I. Strong completeness and
   uniform bounds},
   journal={Ann. of Math. (2)},
   volume={165},
   date={2007},
   number={1},
   pages={171--238},
}

\bib{Petersen:phd-thesis}{book}{
  author={Petersen, H.\,D.},
  title={L2-Betti numbers of locally compact groups},
  year={2012},
  isbn={978-87-7078-993-6},
  publisher={PhD thesis, Department of Mathematical Sciences, Faculty of Science, University of Copenhagen},
  note={\url{http://www.math.ku.dk/noter/filer/phd13hdp.pdf}},
}

\bib{Petersen-Valette:l2-betti-plancherel}{article}{
   author={Petersen, H.\,D.},
   author={Valette, A.},
   title={$L^2$-Betti numbers and Plancherel measure},
   journal={J. Funct. Anal.},
   volume={266},
   date={2014},
   number={5},
   pages={3156--3169},
}

\bib{Platonov-Rapinchuk:algebraic-groups}{book}{
   author={Platonov, V.},
   author={Rapinchuk, A.},
   title={Algebraic groups and number theory},
   series={Pure and Applied Mathematics},
   volume={139},
   note={Translated from the 1991 Russian original by Rachel Rowen},
   publisher={Academic Press, Inc., Boston, MA},
   date={1994},
   pages={xii+614},
}

\bib{Prasad-Raghunathan:metaplectic-kernel}{article}{
   author={Prasad, G.},
   author={Raghunathan, M. S.},
   title={On the congruence subgroup problem: determination of the
   ``metaplectic kernel''},
   journal={Invent. Math.},
   volume={71},
   date={1983},
   number={1},
   pages={21--42},
}

\bib{Prasad-Rapinchuk:computation}{article}{
   author={Prasad, G.},
   author={Rapinchuk, A. S.},
   title={Computation of the metaplectic kernel},
   journal={Inst. Hautes \'Etudes Sci. Publ. Math.},
   number={84},
   date={1996},
   pages={91--187 (1997)},
}


\bib{Raghunathan:on-csp}{article}{
   author={Raghunathan, M. S.},
   title={On the congruence subgroup problem},
   journal={Inst. Hautes \'Etudes Sci. Publ. Math.},
   number={46},
   date={1976},
   pages={107--161},
}

\bib{Reid:profinite-properties}{article}{
   author={Reid, A.\,W.},
   title={Profinite properties of discrete groups},
   conference={
      title={Groups St Andrews 2013},
   },
   book={
      series={London Math. Soc. Lecture Note Ser.},
      volume={422},
      publisher={Cambridge Univ. Press, Cambridge},
   },
   date={2015},
   pages={73--104},
}
		
\bib{Ribes-Zalesskii:profinite-groups}{book}{
   author={Ribes, L.},
   author={Zalesskii, P.},
   title={Profinite groups},
   series={Ergebnisse der Mathematik und ihrer Grenzgebiete. 3. Folge.},
   volume={40},
   publisher={Springer-Verlag, Berlin},
   date={2000},
   pages={xiv+435},
}
\bib{sauer-survey}{article}{
   author={Sauer, R.},
   title={$\ell^2$-Betti numbers of discrete and non-discrete groups},
   book={
          title={New directions in locally compact groups},
          series={London Mathematical Society Lecture Note Series},
          volume={447},
          publisher={Cambridge University Press},
          date={2018}
         },
   pages={205--226},
}

\bib{Serre:arithmetic}{book}{
   author={Serre, J.-P.},
   title={A course in arithmetic},
   note={Translated from the French;
   Graduate Texts in Mathematics, No. 7},
   publisher={Springer-Verlag, New York-Heidelberg},
   date={1973},
   pages={viii+115},
}

\bib{Serre:local-fields}{book}{
   author={Serre, J.-P.},
   title={Local fields},
   series={Graduate Texts in Mathematics},
   volume={67},
   note={Translated from the French by Marvin Jay Greenberg},
   publisher={Springer-Verlag, New York-Berlin},
   date={1979},
   pages={viii+241},
}

\bib{Stucki:master-thesis}{book}{
     author={Stucki, N.},
     title={\(L^2\)-Betti numbers and profinite completions of groups},
     publisher={Master thesis (Karlsruhe Institute of Technology)},
     note={Will be available for download at \url{http://topology.math.kit.edu/21_696.php}},
     date={2018},
}

\bib{Vaserstein:structure}{article}{
   author={Vaser\v ste\u\i n, L. N.},
   title={Structure of the classical arithmetic groups of rank greater than
   $1$},
   language={Russian},
   journal={Mat. Sb. (N.S.)},
   volume={91(133)},
   date={1973},
   pages={445--470, 472},
}
		
		
\bib{witt}{article}{
   author={Witt, E.},
   title={Theorie der quadratischen Formen in beliebigen K\"orpern},
   language={German},
   journal={J. Reine Angew. Math.},
   volume={176},
   date={1937},
   pages={31--44},
}		


\end{biblist}
\end{bibdiv}
\end{document}